\pdfoutput=1
\RequirePackage{ifpdf}
\ifpdf 
\documentclass[pdftex]{sigma}
\else
\documentclass{sigma}
\fi

\numberwithin{equation}{section}

\newtheorem{thm}{Theorem}[section]

\theoremstyle{definition}
\newtheorem{rem}[thm]{Remark}

\begin{document}
\allowdisplaybreaks

\newcommand{\arXivNumber}{1912.03649}

\renewcommand{\thefootnote}{}

\renewcommand{\PaperNumber}{033}

\FirstPageHeading

\ShortArticleName{Nonnegative Scalar Curvature and Area Decreasing Maps}

\ArticleName{Nonnegative Scalar Curvature\\ and Area Decreasing Maps\footnote{This paper is a~contribution to the Special Issue on Scalar and Ricci Curvature in honor of Misha Gromov on his 75th Birthday. The full collection is available at \href{https://www.emis.de/journals/SIGMA/Gromov.html}{https://www.emis.de/journals/SIGMA/Gromov.html}}}

\Author{Weiping ZHANG}

\AuthorNameForHeading{W.~Zhang}

\Address{Chern Institute of Mathematics \& LPMC, Nankai
University, Tianjin 300071, P.R.~China}
\Email{\href{mailto:weiping@nankai.edu.cn}{weiping@nankai.edu.cn}}
\URLaddress{\url{http://old.cim.nankai.edu.cn/members/user/weiping/?name=weiping}}

\ArticleDates{Received December 18, 2019, in final form April 15, 2020; Published online April 22, 2020}

\Abstract{Let $\big(M,g^{TM}\big)$ be a noncompact complete spin Riemannian manifold of even dimension~$n$, with~$k^{TM}$ denote the associated scalar curvature. Let $f\colon M\rightarrow S^{n}(1)$ be a~smooth area decreasing map, which is locally constant near infinity and of nonzero degree. We show that if $k^{TM}\geq n(n-1)$ on the support of~${\rm d}f$, then $ \inf \big(k^{TM}\big)<0$. This answers a~question of Gromov. We use a simple deformation of the Dirac operator to prove the result. The odd dimensional analogue is also presented.}


\Classification{53C27; 57R20; 58J20}

\renewcommand{\thefootnote}{\arabic{footnote}}
\setcounter{footnote}{0}

\section{Introduction} \label{s0}

It is well-known that starting with the famous Lichnerowicz vanishing theorem~\cite{L63}, Dirac operators have played important roles in the study of Riemannian metrics of positive scalar curvature on spin manifolds (cf.~\cite{GL83} and~\cite{LaMi89}). A notable example is Llarull's rigidity theorem~\cite{L98} which states that for a compact spin Riemannian manifold $\big(M,g^{TM}\big)$ of dimension $n$ such that the associated scalar curvature $k^{TM}$ verifies that $k^{TM}\geq n(n-1)$, any (non-strict) area decreasing smooth map $f\colon M\rightarrow S^n(1)$ of nonzero degree is an isometry.

Recently, Gromov states in \cite[p.~45]{Gr19} a noncompact extension of the Llarull theorem. Namely, if $\big(M,g^{TM}\big)$ is an $n$ dimensional noncompact complete spin Riemannian manifold, $f\colon M\rightarrow S^n(1)$ a smooth (non-strict) area decreasing map (which is locally constant near infinity) of nonzero degree such that $k^{TM}\geq n(n-1)$ on the support of~${\rm d}f$, then
\begin{gather}\label{0.1}
\inf \big(k^{TM}\big)\leq 0.
\end{gather}

The argument used by Gromov for~(\ref{0.1}) relies on the relative index theorem of Gromov--Lawson~\cite{GL83}, which depends on the positivity of $k^{TM}$ near infinity. Gromov then raises the question that whether the inequality in (\ref{0.1}) can actually be made {strict}.

The purpose of this short note is to provide a positive answer to this question when~$n$ is even. That is, (\ref{0.1}) can indeed be improved to $\inf \big(k^{TM}\big)< 0$. When $n$ is odd, we improve~(\ref{0.1}) to $\inf \big(k^{TM}\big)< 0$ under the condition that $k^{TM}> n(n-1)$ on the support of~${\rm d}f$.

The main idea of the proof, similar to \cite[equation~(1.11)]{Z19}, is to deform the involved twisted Dirac operator (constructed as in~\cite{L98})
on~$M$ by a suitable endomorphism of the twisted vector bundle (cf.~(\ref{1.7})). The deformed Dirac operator turns out to be invertible near infinity, and one can then apply the relative index theorem to complete the proof.

\section{The main results and their proofs}\label{s1}

This section is organized as follows. In Section \ref{s1.1}, we restate the main results of this paper. In Sections~\ref{s1.2} and~\ref{s1.3}, we prove the main results stated in Section~\ref{s1.1}.

\subsection{The main results}\label{s1.1}

Let $\big(M,g^{TM}\big)$ be an $n$-dimensional noncompact spin complete Riemannian manifold. Let~$k^{TM}$ be the associated scalar curvature. Let $S^n(1)$ be the standard $n$-dimensional unit sphere carrying its canonical metric. Following~\cite{GL83}, a smooth map $f\colon M \rightarrow S^n(1)$ is called area decreasing if for any two form $\alpha\in\Omega^2\big(S^n(1)\big)$, $f^*\alpha\in\Omega^2(M)$ verifies that
\begin{gather}\label{1.1}
|f^*\alpha|\leq |\alpha|.
\end{gather}

We now assume that $f\colon M\rightarrow S^n(1)$ is a smooth area decreasing map such that it is locally constant near infinity. That is, it is locally constant outside a compact subset $K\subset M$. We also assume that
\begin{gather}\label{1.2}
\deg (f)\neq 0.
\end{gather}

Let ${\rm d}f\colon TM\rightarrow TS^n(1)$ be the differential of $f$. The support of ${\rm d}f$ is defined to be $\operatorname{Supp}({\rm d}f)=\overline{\{ x\in M\colon {\rm d}f(x)\neq 0\}}$.

The main result of this short note can be stated as follows.

\begin{thm}\label{t1.1} Under the above assumptions, if $n$ is even and
\begin{gather}\label{1.3}
k^{TM}\geq n(n-1) \qquad {\rm on}\quad \operatorname{Supp}({\rm d}f),
\end{gather}
then one has
\begin{gather}\label{1.4}
\inf \big(k^{TM}\big)<0.
\end{gather}
\end{thm}

When $n$ is odd, we have the following analogue which shows that~(\ref{1.4}) still holds when the inequality (\ref{1.3}) holds strictly.

\begin{thm}\label{t1.2} Under the assumptions above Theorem~{\rm \ref{t1.1}}, if $n$ is odd and
\begin{gather*}
k^{TM}> n(n-1) \qquad {\rm on}\quad \operatorname{Supp}({\rm d}f),
\end{gather*}
then \eqref{1.4} still holds.
\end{thm}

Theorems \ref{t1.1} and \ref{t1.2} will be proved in Sections~\ref{s1.2} and~\ref{s1.3} respectively.

\subsection{Proof of Theorem \ref{t1.1}}\label{s1.2}

Let $S(TM)=S_+(TM)\oplus S_-(TM)$ be the ${\bf Z}_2$-graded Hermitian vector bundle of spinors associated to $\big(TM,g^{TM}\big)$, carrying the canonically induced Hermitian connection $\nabla^{S(TM)}=\nabla^{S_+(TM)}+\nabla^{S_-(TM)}$ (cf.~\cite{LaMi89}). And we use similar notation for $ S^n(1)$.

Following~\cite{L98}, let $E=E_+\oplus E_-$ be the ${\bf Z}_2$-graded Hermitian vector bundle
\begin{gather}\label{1.4b}
f^*\big(S\big(TS^n(1)\big)\big)=f^*\big(S_+\big(TS^n(1)\big)\big)\oplus f^*\big(S_-\big(TS^n(1)\big)\big)
\end{gather}
over $M$ carrying the pull-back Hermitian connection $\nabla^E=\nabla^{E_+}+\nabla^{E_-}$. Let $R^E=\big(\nabla^E\big)^2$ be the curvature of $\nabla^E$.

Let $D^E\colon \Gamma\big(S(TM)\widehat\otimes E\big)\rightarrow
\Gamma\big(S(TM)\widehat\otimes E\big)$ be the canonically defined (twisted by~$E$) Dirac operator (cf. \cite{LaMi89}).\footnote{Here ``$\widehat\otimes$'' is the notation for the ${\bf Z}_2$-graded tensor product (cf.~\cite{Q85}).} Then one has the canonical splitting $D=D_++D_-$ with $D^E_+\colon \Gamma(S_+(TM) \otimes E_+\oplus S_-(TM)\otimes E_-) \rightarrow \Gamma(S_-(TM) \otimes E_+\oplus S_+(TM)\otimes E_-) $, while $D^E_-\colon \Gamma(S_+(TM) \otimes E_-\oplus S_-(TM)\otimes E_+) \rightarrow \Gamma(S_-(TM) \otimes E_-\oplus S_+(TM)\otimes E_+) $. Moreover, one has formally that
\begin{gather*}
\big(D^E_+\big)^*=D^E_-.
\end{gather*}

\looseness=-1 Take any $p\in S^n(1)\setminus {f(M\setminus K)}$. Let $X\in TS^n(1)$ be a smooth vector field on $S^n(1)$ such that $|X|>0$ on $S^n(1)\setminus \{p\}$. Let $v=c(X)\colon S_+\big(TS^n(1)\big)\rightarrow S_-\big(TS^n(1)\big)$ be the Clifford action of $X$. Let $v^*\colon S_-\big(TS^n(1)\big)\rightarrow S_+\big(TS^n(1)\big)$ be the adjoint of $v$ with respect to the Hermitian metrics on $S_\pm(TS^n(1))$. Let $V\colon S\big(TS^n(1)\big)\rightarrow S\big(TS^n(1)\big)$ be the self-adjoint odd endomorphism defined by
\begin{gather}\label{1.6}
V=v+v^*.
\end{gather}
Then one has
\begin{gather}\label{1.6a}V^2=|X|^2.
\end{gather}
Thus $V$ is invertible on $ {f(M\setminus K)} = \overline {f(M\setminus K)} $. Also, $f^*V$ extends to an action on $S(TM)\widehat\otimes E$ such that for any $\alpha\in S(TM)$, $u\in E$, one has $(f^*V )\big(\alpha\widehat \otimes u\big)=(-1)^{\deg (\alpha)}\alpha\widehat\otimes (f^*V)(u)$ (cf.~\cite{Q85}).

Let $U_{\frac{1}{2}}\subset M$ be the subset defined by
$U_{\frac{1}{2}}=\big\{ x\in \operatorname{Supp}({\rm d}f)\colon |{\rm d}f(x)|<\frac{1}{2}\big\}$. Let $V_{\frac{1}{2}}\subset M$ be the open subset defined by $V_{\frac{1}{2}}=\big\{ x \colon \big|{\wedge}^2({\rm d}f(x))\big|>\frac{1}{2}\big\}$, where $\wedge^2({\rm d}f)$ is the induced action of ${\rm d}f$ on the exterior product $\wedge^2(TM)$. Clearly, $\overline{U}_{\frac{1}{2}}\cap \overline{V}_{\frac{1}{2}}=\varnothing$.

Let $\varphi\colon M\rightarrow [0,1]$ be a smooth function such that $\varphi=1$ on $(M\setminus \operatorname{Supp}({\rm d}f) )\cup U_{\frac{1}{2}}$, while $\varphi=0$ on~$V_{\frac{1}{2}}$. The existence of $\varphi$ is clear.

Similar to \cite[equation~(1.11)]{Z19}, for any $\varepsilon>0$, let $D^E_\varepsilon\colon \Gamma\big(S(TM)\widehat\otimes E\big)\rightarrow
\Gamma\big(S(TM)\widehat\otimes E\big)$ be the deformed twisted Dirac operator defined by\footnote{In view of~\cite{Q85}, one may regard $D^E_\varepsilon$ as a ``super'' Dirac operator.}
\begin{gather}\label{1.7}
D^E_\varepsilon =D^E+\varepsilon \varphi f^*V.
\end{gather}
Let $D^E_{\varepsilon,+}\colon \Gamma(S_+(TM) \otimes E_+\oplus S_-(TM)\otimes E_-) \rightarrow \Gamma(S_-(TM) \otimes E_+\oplus S_+(TM)\otimes E_-) $ and $D^E_{\varepsilon,-}\colon \Gamma(S_+(TM) \otimes E_-\oplus S_-(TM)\otimes E_+) \rightarrow \Gamma(S_-(TM) \otimes E_-\oplus S_+(TM)\otimes E_+) $ be the natural restrictions.

From (\ref{1.7}), one has
\begin{gather}\label{1.8}
\big(D^E_\varepsilon\big)^2 =\big(D^E\big)^2+\varepsilon c({\rm d}\varphi) f^*V
+\varepsilon \varphi\big[D^E,f^*V\big]+\varepsilon^2\varphi^2f^*\big(V^2\big),
\end{gather}
where $c(\cdot)$ is the notation for the Clifford action, $[\cdot,\cdot]$ is the notation for the supercommutator in the sense of~\cite{Q85}, and we identify~${\rm d}\varphi$ with the gradient of $\varphi$.

Let $e_1, \dots, e_n$ be an orthonormal basis of $\big(TM,g^{TM}\big)$. By the definition of the Dirac operator, $D^E=\sum\limits_{i=1}^n c(e_i)\nabla_{e_i}$, one has (cf.~\cite{Z19})
\begin{gather}\label{1.13}
\big[D^E,f^*V\big] =\sum_{i=1}^nc (e_i ) f^* \big(\nabla^{S(TS^n(1))}_{f_*(e_i)}V\big).
\end{gather}

Since by definition $\varphi=1$ on $M\setminus K$, while $f$ is locally constant on $M\setminus K $, from (\ref{1.8}) and~(\ref{1.13}) we see that the following identity holds on $M\setminus K$,
\begin{gather}\label{1.9}
\big(D^E_\varepsilon\big)^2 =\big(D^E\big)^2 +\varepsilon^2 f^*\big(V ^2\big).
\end{gather}

As $V|_{ {f(M\setminus K)}}$ is invertible, from (\ref{1.9}), one sees that there is a constant $a>0$ such that for any $s\in \Gamma\big(S(TM)\widehat\otimes E\big)$ supported in a~compact subset of $M\setminus K$, one has
\begin{gather}\label{1.10}
\big\|D^E_\varepsilon s\big\| \geq \varepsilon a \|s\|.
\end{gather}

 From (\ref{1.10}), one sees that one can apply the relative index theorem in~\cite{GL83} to $D^E_{\varepsilon,+}$ (compare with \cite[Section~6$\frac{4}{5}$]{Gr96}). In particular, one gets, by similar computations as in~\cite{GL83} and \cite[Proposition~III.11.24]{LaMi89},
\begin{align}
\operatorname{ind}\big(D^E_{\varepsilon,+}\big)
&= \deg (f)\big\langle {\rm ch}\big(S_+\big(TS^n(1)\big)\big) -
 {\rm ch}\big(S_-\big(TS^n(1)\big)\big),\big[S^n(1)\big]\big\rangle
\nonumber\\
& =(-1)^{\frac{n}{2}}\deg (f) \chi\big(S^n(1)\big)
 =2(-1)^{\frac{n}{2}}\deg (f) .\label{1.11}
\end{align}

 By the Lichnerowicz formula \cite{L63} for $D^E$ (cf.~\cite{LaMi89}), one has
\begin{gather}\label{1.12}
\big(D^E\big)^2=-\Delta^E+\frac{k^{TM}}{4}+\frac{1}{2}\sum_{i, j=1}^n c (e_i )c (e_j )R^E (e_i,e_j ) ,
\end{gather}
where $-\Delta^E\geq 0$ is the corresponding Bochner Laplacian.

From \cite[equation~(4.6)]{L98}, one knows that
\begin{gather}\label{1.14}
\frac{1}{2}\sum_{i, j=1}^n c (e_i )c (e_j )R^E (e_i,e_j )\geq -\frac{n(n-1)}{4}\big|{\wedge}^2({\rm d}f)\big|.
\end{gather}

From (\ref{1.1}), (\ref{1.3}), (\ref{1.8}), (\ref{1.12}) and (\ref{1.14}), one has that near any $x\in V_{\frac{1}{2}}$,
\begin{gather}\label{1.15}
\big(D_\varepsilon^E\big)^2+\Delta^E \geq0.
\end{gather}

Near any $x\in \operatorname{Supp}({\rm d}f)\setminus \overline{V}_{\frac{1}{2}}$, by
(\ref{1.3}), (\ref{1.8}), (\ref{1.12}) and (\ref{1.14}), one has
\begin{gather}\label{1.16}
\big(D_\varepsilon^E\big)^2+\Delta^E \geq \frac{n(n-1)}{8}+\varepsilon c({\rm d}\varphi)V+\varepsilon \varphi \big[D^E,f^*V\big]+\varepsilon^2\varphi^2f^*\big(V^2\big).
\end{gather}

From (\ref{1.8}), (\ref{1.13}), (\ref{1.12}) and (\ref{1.14}), one has that near any $x\in M\setminus \operatorname{Supp}({\rm d}f)$,
\begin{gather}\label{1.18}
\big(D_\varepsilon^E\big)^2+\Delta^E \geq \frac{k^{TM}}{4}+\varepsilon^2f^*\big(V^2\big).
\end{gather}

Now assume that (\ref{1.4}) does not hold. Then one has over $M$ that
 \begin{gather}\label{1.19}
k^{TM}\geq 0.
\end{gather}

From (\ref{1.13}), (\ref{1.9}), (\ref{1.15})--(\ref{1.19}) and the compactness of $\operatorname{Supp}({\rm d}f)$, we see that when $\varepsilon>0$ is small enough, there is a smooth nonnegative endormorphism $a_\varepsilon$ of $S(TM)\widehat\otimes E$ such that
 \begin{gather}\label{1.20}
a_\varepsilon>0 \qquad {\rm on}\quad (M\setminus K)\cup U_{\frac{1}{2}}
\end{gather}
and that one has on $M$ that
\begin{gather}\label{1.21}
\big(D_\varepsilon^E\big)^2 \geq -\Delta^E +a_\varepsilon.
\end{gather}

From (\ref{1.20}) and (\ref{1.21}), one finds that the relative index of $D^E_{\varepsilon,+}$ verifies that
\begin{gather*}
\operatorname{ind}\big(D^E_{\varepsilon,+}\big)=0,
\end{gather*}
which contradicts (\ref{1.2}) and (\ref{1.11}). This completes the proof of Theorem~\ref{t1.1}.

\begin{rem}\label{t1.3}
In view of (\ref{1.10}) and \cite[Section~6$\frac{4}{5}$]{Gr96}, one sees that the above proof fits with Gromov's suggestion in \cite[p.~45]{Gr19} that one may use the Callias type index arguments to deal with (\ref{1.4}). Also, from the above proof one sees that~(\ref{1.3}) can be weakened to
\begin{gather}\label{1.23}
k^{TM}\geq n(n-1)\big|{\wedge}^2({\rm d}f)\big| \qquad {\rm on}\quad \operatorname{Supp}({\rm d}f),
\end{gather}
with the inequality being strict on $\operatorname{Supp}({\rm d}f)\setminus \{x\in M\colon {\rm d}f(x)\neq 0\}$. With~(\ref{1.23}) the condition~(\ref{1.1}) is no longer needed.
\end{rem}

\subsection{Proof of Theorem \ref{t1.2}}\label{s1.3}

In view of Remark \ref{t1.3}, we will state and prove the following refined version of Theorem \ref{t1.2}.

\begin{thm}\label{t1.4} Let $\big(M,g^{TM}\big)$ be a noncompact spin complete Riemannian manifold of odd dimension $n$, and $f\colon M\rightarrow S^n(1)$ be a smooth map which is locally constant near infinity and of nonzero degree. If we assume that the scalar curvature $k^{TM}$ of $g^{TM}$ verifies that
\begin{gather}\label{1.24}
k^{TM}> n(n-1)\big|{\wedge}^2({\rm d}f)\big| \qquad {\rm on}\quad \operatorname{Supp}({\rm d}f),
\end{gather}
then one has
\begin{gather}\label{1.25}
\inf\big(k^{TM}\big)<0.
\end{gather}
\end{thm}

\begin{proof} For any $R>0$, let $S^1(R)$ be the round circle of radius $R$, with the canonical met\-ric~${\rm d}t^2$. Let $M\times S^1(R)$ be the complete Riemannian manifold of the product metric $g^{TM}\oplus {\rm d}t^2$. Follo\-wing~\cite{L98}, we consider the chain of maps
\begin{gather*}
 M\times S^1(R)\xrightarrow{f\times \frac{1}{R}} S^n(1)\times S^1(1) \xrightarrow{h} S^{n+1}(1),
\end{gather*}
where $\frac{1}{R}\colon S^1(R)\rightarrow S^1(1)$ is the standard shrinking map, and~$h$ is a suspension map of degree one such that $|{\rm d}h|\leq 1$. Let $f_R=h\circ \big(f\times \frac{1}{R}\big)\colon M\times S^1(R)\rightarrow S^{n+1}(1)$ denote the composition. Then one has\footnote{The degree of $f_R$ is well-defined as~$f_R$ has no regular point near infinity.}
\begin{gather}\label{1.27}
 \deg (f_R)=\deg (f)\neq 0.
\end{gather}

Let $p\in S^{n+1}(1) $ be any regular value of $f_R$. Let $X\in TS^{n+1}(1)$ be a smooth vector field on $S^{n+1}(1)$ such that $|X|>0$ on $S^{n+1}(1)\setminus \{p\}$.
Let $v=c(X)\colon S_+\big(TS^{n+1}(1)\big)\rightarrow S_-\big(TS^{n+1}(1)\big)$ be the Clifford action of~$X$. Let $V\colon S\big(TS^{n+1}(1)\big)\rightarrow S\big(TS^{n+1}(1)\big)$ be defined as in~(\ref{1.6}). Then~(\ref{1.6a}) holds and~$V$ is invertible on~$ S^{n+1}(1)\setminus \{p\}$. In particular, there is $\delta>0$ such that
\begin{gather}\label{1.29}
V^2 \geq \delta \qquad {\rm on}\quad \overline{f_R((M\setminus \operatorname{Supp}({\rm d}f))\times S^1(R)) }.
\end{gather}

For simplicity, denote $M_R=M\times S^1(R)$. Let $E_R =f_R^*\big(S\big(TS^{n+1}(1)\big)\big)$ be the ${\bf Z}_2$-graded Hermitian vector bundle over $M_R$ as in~(\ref{1.4b}).
Let $D^{E_R}\colon \Gamma\big(S(TM_R)\widehat\otimes E_R\big)\rightarrow
\Gamma\big(S(TM_R)\widehat\otimes E_R\big)$ denote the canonical twisted Dirac operator.

As in (\ref{1.7}), for any $\varepsilon>0$,
let $D^{E_R}_\varepsilon\colon \Gamma\big(S(TM_R)\widehat\otimes E_R\big)\rightarrow
\Gamma\big(S(TM_R)\widehat\otimes E_R\big)$ be the deformed twisted Dirac operator defined by\footnote{There is no need here to introduce the function~$\varphi$ as in~(\ref{1.7}), as the inequality in~(\ref{1.24}) is strict.}
\begin{gather*}
D^{E_R}_\varepsilon =D^{E_R}+\varepsilon f_R^*V.
\end{gather*}
Let $D^{E_R}_{\varepsilon,+}\colon \Gamma (S_+(TM_R) \otimes E_{R,+}\oplus S_-(TM_R)\otimes E_{R,-} ) \rightarrow \Gamma(S_-(TM_R) \otimes E_{R,+}\oplus S_+(TM_R)\otimes E_{R,-}) $ be the natural restriction.

The analogue of (\ref{1.8}) now takes the form
\begin{gather}\label{1.31}
\big(D^{E_R}_\varepsilon\big)^2 =\big(D^{E_R}\big)^2
+\varepsilon \big[D^{E_R},f_R^*V\big]+\varepsilon^2 f_R^*\big(V^2\big).
\end{gather}

Let $K$ be a compact subset of $M$ such that $f\colon M\rightarrow S^n(1)$ is locally constant on $M\setminus K$.

From (\ref{1.13}) and the definition of $f_R$, one finds on $ (M\setminus K)\times S^1(R)$ that
\begin{gather}\label{1.32}
\big[D^{E_R},f_R^*V\big] = O\left(\frac{1}{R}\right).
\end{gather}

From (\ref{1.29}), (\ref{1.31}) and (\ref{1.32}), one sees that for any $\varepsilon>0$, there exist $R_0>0$ and \mbox{$b>0$} such that when $R\geq R_0$, for any $s\in \Gamma\big(S(TM_R)\widehat\otimes E_R\big)$ supported in a compact subset of $ (M\setminus K)\times S^1(R)$, one has
\begin{gather}\label{1.33}
\big\|D^{E_R}_\varepsilon s\big\| \geq \varepsilon b \|s\|.
\end{gather}

One can then apply the relative index theorem \cite{GL83} to $D^{E_R}_{\varepsilon,+}$ and get as in (\ref{1.11}) that
\begin{gather}\label{1.34}
\operatorname{ind}\big(D^{E_R}_{\varepsilon,+}\big) = 2(-1)^{\frac{n+1}{2}}\deg (f_R ) .
\end{gather}

Let $e_1, \dots, e_{n+1}$ be an orthonormal basis of $TM_R$. The Lichnerowicz formula (\ref{1.12}) now takes the form
\begin{gather}\label{1.35}
\big(D^{E_R}\big)^2=-\Delta^{E_R}+\frac{\pi_R^*k^{TM}}{4}+\frac{1}{2}\sum_{i, j=1}^{n+1}
c (e_i )c (e_j )R^{E_R} (e_i,e_j ) ,
\end{gather}
where $\pi_R\colon M\times S^1(R)\rightarrow M$ denotes the natural projection.

By \cite[p.~68]{L98}, one has at any $(x,y)\in M\times S^1(R)$ that
\begin{gather}\label{1.36}
 \frac{1}{2}\sum_{i, j=1}^{n+1}
c (e_i )c (e_j )R^{E_R} (e_i,e_j ) \geq -\frac{n(n-1)}{4}\big|{\wedge}^2({\rm d}f(x))\big| + |{\rm d}f(x)| O\left(\frac{1}{R}\right).
\end{gather}

From (\ref{1.13}), (\ref{1.24}), (\ref{1.36}) and the compactness of $\operatorname{Supp}({\rm d}f)\subset M$, one sees that there is $\eta>0$ such that when $\varepsilon>0$ is small enough, the following formula holds on $(\operatorname{Supp}({\rm d}f))\times S^1(R)$,
\begin{gather}\label{1.37}
 \frac{\pi_R^*k^{TM}}{4}+\frac{1}{2}\sum_{i, j=1}^{n+1}
c(e_i)c(e_j)R^{E_R}(e_i,e_j) + \varepsilon \big[D^{E_R},f_R^*V\big]
\geq \eta+ (\pi_R^* |{\rm d}f| )O\left(\frac{1}{R}\right).
\end{gather}

Now we observe that (\ref{1.32}) indeed holds on $(M\setminus \operatorname{Supp}({\rm d}f))\times S^1(R)$.

From (\ref{1.29}), (\ref{1.31}), (\ref{1.32}) and (\ref{1.35})--(\ref{1.37}), one finds that if~(\ref{1.25}) does not hold, then one can first take $\varepsilon>0$ small enough, and then take $R>0$ large enough to get a positive endomorphism $a_{\varepsilon,R}$ of $S(TM_R)\widehat\otimes E_R$ such that in addition to~(\ref{1.33}), the following formula also holds,
 \begin{gather}\label{1.38}
\big(D_\varepsilon^{E_R}\big)^2 \geq -\Delta^{E_R} +a_{\varepsilon,R}.
\end{gather}

From (\ref{1.38}), one gets
\begin{gather*}
\operatorname{ind}\big(D^{E_R}_{\varepsilon,+}\big)=0,
\end{gather*}
which contradicts (\ref{1.27}) and (\ref{1.34}). This completes the proof of Theorem~\ref{t1.4}.
\end{proof}

\subsection*{Acknowledgements}

The author would like to thank the referees for careful readings and very helpful suggestions. This work was partially supported by NNSFC Grant no.~11931007.

\pdfbookmark[1]{References}{ref}
\LastPageEnding

\end{document}